\documentclass[12pt]{amsart}
\usepackage{amssymb,amsmath}
\usepackage{geometry}    
\usepackage{mathrsfs}

\usepackage{vmargin}
\setmargrb{1in}{1in}{1in}{1in}

\newtheorem{theorem}{Theorem}[section]
\newtheorem{lemma}[theorem]{Lemma}

\theoremstyle{definition}

\newtheorem{remark}[theorem]{Remark}

\begin{document}

\title{Approximation to an extremal number, its square and its cube}

\author{Johannes Schleischitz} 

\address{Institute of Mathematics, Boku Vienna, Austria  \\ 
johannes.schleischitz@boku.ac.at}

\begin{abstract}
We study rational approximation properties for successive powers of extremal numbers
defined by Roy. For $n\in{\{1,2\}}$, the classic approximation constants 
$\lambda_{n}(\zeta),\widehat{\lambda}_{n}(\zeta),w_{n}(\zeta),\widehat{w}_{n}(\zeta)$ 
connected to an extremal number $\zeta$ have been established and in fact much more is known. 
However, so far almost nothing had been known for $n\geq 3$.
In this paper we determine all classic approximation constants as above for $n=3$. Our methods
will more generally provide detailed information on the combined graph 
defined by Schmidt and Summerer assigned to an extremal
number, its square and its cube.
We provide some results for $n=4$ as well. In the course of the proofs
of the main results we establish a very general connection between Khintchine's transference 
inequalities and uniform approximation. 
\end{abstract}

\maketitle

{\footnotesize{Supported by the Austrian Science Fund FWF grant P24828.} \\

Math Subject Classification 2010: 11H06, 11J13}

\vspace{4mm}

\section{Approximation constants and extremal numbers} \label{sek1}

Let $\zeta$ be a real transcendental number and $n\geq 1$ be an integer. For $1\leq j\leq n+1$
we define the approximation constants $\lambda_{n,j}(\zeta)$ as
the supremum of $\eta\in{\mathbb{R}}$ such that the system 
\begin{equation}  \label{eq:lambda}
\vert x\vert \leq X, \qquad \max_{1\leq i\leq n} \vert \zeta^{i}x-y_{i}\vert \leq X^{-\eta},  
\end{equation}
has (at least) $j$ linearly independent solutions $(x,y_{1},y_{2},\ldots, y_{n})\in{\mathbb{Z}^{n+1}}$ 
for arbitrarily large values of $X$. Moreover, let $\widehat{\lambda}_{n,j}(\zeta)$
be the supremum of $\eta$
such that \eqref{eq:lambda} has (at least) $j$ linearly independent solutions for 
all sufficiently large $X$. 
In case of $j=1$ we also only write $\lambda_{n}(\zeta)$ and $\widehat{\lambda}_{n}(\zeta)$ respectively,
which are just the classical approximation constants defined by Bugeaud and Laurent~\cite{buglau}.
By Dirichlet's Theorem for all transcendental real $\zeta$ and $n\geq 1$ these exponents satisfy the estimate
\begin{equation} \label{eq:ldiri}
\lambda_{n}(\zeta)\geq \widehat{\lambda}_{n}(\zeta)\geq \frac{1}{n}.
\end{equation}
Moreover from the definition we see that
\[
\lambda_{1}(\zeta)\geq \lambda_{2}(\zeta)\geq \cdots, 
\qquad \widehat{\lambda}_{1}(\zeta)\geq \widehat{\lambda}_{2}(\zeta)\geq \cdots.
\]

Similarly, let $w_{n,j}(\zeta)$ and $\widehat{w}_{n,j}(\zeta)$ be the supremum of  
$\eta\in{\mathbb{R}}$ such that the system 
\begin{equation}  \label{eq:w}
H(P) \leq X, \qquad  0<\vert P(\zeta)\vert \leq X^{-\eta},  
\end{equation}
has (at least) $j$ linearly independent polynomial solutions 
$P(T)=a_{n}T^{n}+a_{n-1}T^{n-1}+\cdots+a_{0}$ of degree at most $n$
with integers $a_{j}$ for arbitrarily large $X$ and all large $X$ respectively, where 
$H(P)=\max_{0\leq j\leq n} \vert a_{j}\vert$. Again for $j=1$ we also write $w_{n}(\zeta)$ and $\widehat{w}_{n}(\zeta)$
which coincide with classical exponents. Again by Dirichlet's Theorem we have
\begin{equation} \label{eq:wmono}
w_{n}(\zeta)\geq \widehat{w}_{n}(\zeta)\geq n.
\end{equation}
Moreover it is obvious that
\[
w_{1}(\zeta)\leq w_{2}(\zeta)\leq \cdots, 
\qquad \widehat{w}_{1}(\zeta)\leq \widehat{w}_{2}(\zeta)\leq \cdots. 
\]
The exponents defined above are connected via Khintchine's transference inequalities~\cite{khintchine} 
\begin{equation} \label{eq:khintchine}
\frac{w_{n}(\zeta)}{(n-1)w_{n}(\zeta)+n}\leq \lambda_{n}(\zeta)\leq \frac{w_{n}(\zeta)-n+1}{n}.
\end{equation}
Similarly thanks to German~\cite{german} we know that the uniform exponents are connected via
\begin{equation} \label{eq:ogerman}
\frac{\widehat{w}_{n}(\zeta)-1}{(n-1)\widehat{w}_{n}(\zeta)}\leq
 \widehat{\lambda}_{n}(\zeta)\leq \frac{\widehat{w}_{n}(\zeta)-n+1}{\widehat{w}_{n}(\zeta)}. 
\end{equation}
We point out that the estimates \eqref{eq:khintchine} and \eqref{eq:ogerman}
hold more generally for the analogue exponents concerning
vectors $\underline{\zeta}\in\mathbb{R}^{n}$ whose coordinates
are $\mathbb{Q}$-linearly independent together with $\{1\}$,
see for example~\cite{ss}. This will be of some importance in Remark~\ref{kuh}. 
Moreover in this case all estimates in \eqref{eq:khintchine} and \eqref{eq:ogerman} are known to be optimal.

It is known due to Davenport and Schmidt~\cite{davsh}
that $\widehat{w}_{2}(\zeta)\leq (3+\sqrt{5})/2$ for all real transcendental $\zeta$.
Roy~\cite{royyy} proved that there exist countably many real transcendental numbers for which equality holds,
and called such numbers extremal numbers. Their approximation properties have been intensely studied
in dimensions $n\in\{1,2\}$. We gather below some of the known facts which will be of importance for this paper. Throughout the paper let
\[
\rho=2+\sqrt{5}, \qquad \tau=\frac{3+\sqrt{5}}{2}, \qquad \nu=\frac{1+\sqrt{5}}{2}, \qquad \gamma=\frac{\sqrt{5}-1}{2}.
\]
These values are linked via $\tau=\nu^{2}, \rho=\nu^{3}$ and $\gamma=\nu^{-1}$. 
Moreover $\tau=\nu+1$ and $\nu^{2}-\nu-1=0$.
It is known that for $\zeta$ an extremal number the identities
\begin{equation} \label{eq:werte}
w_{1}(\zeta)=\lambda_{1}(\zeta)=\lambda_{2}(\zeta)=1, \quad \widehat{\lambda}_{2}(\zeta)=\gamma,
\quad w_{2}(\zeta)=\rho, \quad \widehat{w}_{2}(\zeta)=\tau
\end{equation}
hold. Concerning the higher successive minima functions
it is immediate by Roy's results that any extremal number satisfies
\begin{align} 
w_{2,2}(\zeta)&=\tau, \qquad w_{2,3}(\zeta)=\nu, \qquad \lambda_{2,2}(\zeta)=\gamma, 
\qquad \lambda_{2,3}(\zeta)=\gamma^{2},
 \label{eq:rhotau} \\
\widehat{w}_{2,2}(\zeta)&=\nu, \qquad
\widehat{w}_{2,3}(\zeta)=1, \qquad \widehat{\lambda}_{2,2}(\zeta)=\gamma^{2}, 
\qquad \widehat{\lambda}_{2,3}(\zeta)=\gamma^{3}. \label{eq:taurho}
\end{align}

In fact even more detailed approximation properties are known for $n=2$. 
There is concise information on the integral approximation vectors 
inducing very good approximations in \eqref{eq:lambda} 
such as for the polynomials inducing very good approximations in \eqref{eq:w}. We will concretely utilize
the following consequence of Roy's results
which is part of the claim of~\cite[Theorem~7.2]{roy}. 
See also~\cite[Proposition~8.1, Theorem~8.2]{royyy}.
As usual $a\asymp b$ means
both $a\ll b$ and $b\ll a$ are satisfied everywhere it occurs in the sequel. 

\begin{theorem}[Roy] \label{roythm}
For any extremal number $\zeta$ there exists a sequence of irreducible
polynomials $(P_{k})_{k\geq 1}\in\mathbb{Z}[T]$ of degree precisely two such that
\[
H(P_{k+1})\asymp H(P_{k})^{\nu}, \qquad  \vert P_{k}(\zeta)\vert \asymp H(P_{k})^{-\rho}.
\]
Moreover we have 
\begin{equation}  \label{eq:strich}
\vert P_{k}^{\prime}(\zeta)\vert \asymp H(P_{k}).
\end{equation}
All the implied constants depend on $\zeta$ only.
\end{theorem}

For the irreducibility and \eqref{eq:strich} see~\cite[Proposition~8.1, Theorem~8.2]{royyy},
the other claims are part of the claims of~\cite[Theorem~7.2]{roy}.
In fact the irreducibility is easily deduced from $\lambda_{1}(\zeta)=1$ in \eqref{eq:werte} 
and \eqref{eq:multipli} below. Indeed these relations imply that $P_{k}$ in the theorem 
cannot have a rational root at least for large $k$ and are thus indeed irreducible. 
In context of \eqref{eq:rhotau}, \eqref{eq:taurho} we finally mention that for
$n=2$ extremal numbers induce the regular graph defined by Schmidt and Summerer, we 
refer to~\cite{sums}.

This paper aims to provide a better understanding of the classic approximation constants
for extremal numbers in higher dimension $n>2$. More generally we will provide a description of 
the behavior of the approximation functions $L_{j}(q)$ and $L_{j}^{\ast}(q)$ defined by Schmidt and Summerer~\cite{ss}
in the course of their study of parametric geometry of numbers, for $n=3$ and
partially for $n=4$. We recall basic facts on parametric
geometry of numbers in Section~\ref{para}. 
Our results will arise as a combination of the known results on extremal numbers for $n\in\{1,2\}$
recalled above with estimates from parametric geometry of numbers. 
So far only few non-trivial quantitative results 
on classical approximation constants for extremal numbers in dimension $n>2$ exist.
The estimates 
\[
w_{n}(\zeta)\leq \exp\{c(\zeta)\cdot (\log (3n))^{2}(\log\log(3n))^{2}\}
\]
for all $n\geq 1$ and some constant $c(\zeta)>0$ are due to Adamczewski and Bugeaud~\cite{adabu}.
It was recently proved~\cite{buschl} that $\widehat{w}_{3}(\zeta)\leq 4$ for extremal    
numbers $\zeta$, which improves the upper bound $3+\sqrt{2}$ valid for 
all transcendental real $\zeta$     
from the same paper (which in turn improved the bound $2n-1=5$ of
Davenport and Schmidt~\cite[Theorem~2b]{davsh}).    
However, we will determine the precise value of $\widehat{w}_{3}(\zeta)$ in Theorem~\ref{ndrei}. 
Besides approximation to extremal numbers by cubic algebraic integers has been investigated.
Roy~\cite{royyy} showed that for extremal number $\zeta$ and 
any algebraic integer $\alpha$ of degree three we have
\[
\vert\zeta-\alpha\vert \gg H(\alpha)^{-\tau-1}.
\]
Moreover in~\cite[Theorem~1.1]{roydam} he showed that for some extremal 
numbers the exponent $-1-\tau$ can be replaced by $-\tau$. 
The exponent $-\tau$ is optimal since
\[
\vert\zeta-\alpha\vert \ll H(\alpha)^{-\tau}
\]
has solutions in algebraic integers $\alpha$ of degree at most three and arbitrarily
large height $H(\alpha)$ for any given real number $\zeta$, as shown by Davenport and Schmidt~\cite{davsh}.
It follows that for any real $\zeta$ there are monic polynomials of degree at most three
and arbitrarily large height $H(P)$ such that
\[
\vert P(\zeta)\vert \ll H(P)^{-\nu}.
\]
It follows from~\cite{royyy} that the exponent $\-\nu$ is optimal as well, since again 
the reverse inequality holds at least for some class of extremal numbers and arbitrarily large $H(P)$.

\section{New results}

\subsection{The case $n=3$} \label{sek31}

The first major result of the paper is the following.

\begin{theorem} \label{ndrei}
Let $\zeta$ be an extremal number.
Then we have
\begin{equation} \label{eq:zweite}
w_{3}(\zeta)=w_{2}(\zeta)=\rho, \qquad \lambda_{3}(\zeta)=\frac{1}{\sqrt{5}},
\end{equation}
and
\begin{equation} \label{eq:erste}
\widehat{w}_{3}(\zeta)=3, \qquad \widehat{\lambda}_{3}(\zeta)=\frac{1}{3}.
\end{equation}

\end{theorem}

See the comments subsequent to Lemma~\ref{technisch} 
below for additional information on the dynamic
behavior of the successive minima as parametric functions. This dynamical point of view
will also enable us to derive the following Theorem~\ref{sternchen} from Theorem~\ref{ndrei}. As usual for an
algebraic number $\alpha$ we write $H(\alpha)=H(P)$ where $P\in\mathbb{Z}[T]$ is
the irreducible minimal polynomial of $\alpha$ over $\mathbb{Z}[T]$ with coprime coefficients.  

\begin{theorem} \label{sternchen}
Let $\zeta$ be an extremal number and $\epsilon>0$. Then the estimate
\begin{equation} \label{eq:trio}
\vert Q(\zeta)\vert \leq H(Q)^{-3-\epsilon}
\end{equation}
has only finitely many irreducible solutions $Q\in{\mathbb{Z}[T]}$ of degree precisely three.
In particular
\begin{equation} \label{eq:genaudrei1}
\vert \zeta-\alpha\vert \leq H(\alpha)^{-4-\epsilon} 
\end{equation}
has only finitely many algebraic solutions $\alpha$ of degree precisely three.
On the other hand the estimates
\begin{equation} \label{eq:otherhand}
\vert Q(\zeta)\vert \leq H(Q)^{-3+\epsilon}, \qquad \vert \zeta-\alpha\vert \leq H(\alpha)^{-4+\epsilon} 
\end{equation}
have solutions in irreducible polynomials
$Q$ of degree precisely three and algebraic $\alpha$ of degree precisely three of 
arbitrarily large heights $H(Q)$ and $H(\alpha)$. 
Moreover there are arbitrarily large $X$ such that
\begin{equation} \label{eq:trio2}
H(Q)\leq X, \qquad \vert Q(\zeta)\vert \leq X^{-\sqrt{5}-\epsilon}
\end{equation}
has no irreducible solution $Q\in{\mathbb{Z}[T]}$ of degree precisely three.  
In particular for arbitrarily large $X$ the system
\begin{equation} \label{eq:genaudrei2}
H(\alpha)\leq X, \qquad \vert \zeta-\alpha\vert \leq H(\alpha)^{-1}X^{-\sqrt{5}-\epsilon} 
\end{equation}
has no algebraic solution $\alpha$ of degree precisely three.
\end{theorem}   

We strongly expect that the exponents in \eqref{eq:trio2} and \eqref{eq:genaudrei2}
are optimal as well. See the comments below the proof of Theorem~\ref{sternchen}
for a heuristic argument that supports this belief. Compare Theorem~\ref{sternchen} with 
the estimates concerning approximation by algebraic integers $\alpha$ at the end of
Section~\ref{sek1}.

\subsection{The case $n=4$}

We want to establish a lower bound for the exponent $\lambda_{4}(\zeta)$. Our result,
based on parametric geometry of numbers, is the following.

\begin{theorem} \label{nvier}
Let $\zeta$ be an extremal number. Then
\begin{equation} \label{eq:lammda}
\lambda_{4}(\zeta)\geq \frac{\gamma}{2}=\frac{\sqrt{5}-1}{4}.
\end{equation}
If $w_{4}(\zeta)=w_{2}(\zeta)=\rho$ then there is equality in \eqref{eq:lammda} and moreover
\begin{equation} \label{eq:uniformiert}
\widehat{w}_{4}(\zeta)=4, \qquad \widehat{\lambda}_{4}(\zeta)=\frac{1}{4}.
\end{equation}

\end{theorem}

Observe that $\rho\approx 4.2361>4$, so the assumption of the conditioned results are natural
and thus we believe that there is actually equality in \eqref{eq:lammda} and that \eqref{eq:uniformiert} holds.
The results of Section~\ref{sek31} also support this belief.
On the other hand \eqref{eq:wmono} prohibits $w_{n}(\zeta)=\rho$ for $n\geq 5$,
which in general prohibits the methods of the paper to work for $n\geq 5$.

The constant in \eqref{eq:lammda} is approximately $\gamma/2\approx 0.3090$. 
Observe that this improves the lower bound derived from $w_{4}(\zeta)\geq w_{2}(\zeta)=\rho$
in combination with Khintchine's transference inequalities \eqref{eq:khintchine},
which turns out to be $(2+\sqrt{5})/(10+3\sqrt{5})\approx 0.2535$, only slightly larger than
the trivial bound $1/4$ from \eqref{eq:ldiri}. 

\section{Preparatory results}

\subsection{Parametric geometry of numbers} \label{para}

For the proofs of the new results
we introduce some concepts of the parametric geometry of numbers 
following Schmidt and Summerer~\cite{ss},~\cite{ssch}, where we develop
the theory only as far as it is needed for our purposes and slightly 
deviate from their notation. In particular we restrict
to the case of successive powers. Some more specific properties
will be carried out in Section~\ref{proofs} for immediate application
to preliminary results.
Let $\zeta\in{\mathbb{R}}$ be given and
$Q>1$ a parameter. For $n\geq 1$ and $1\leq j\leq n+1$
define $\psi_{n,j}(Q)$ as the minimum of $\eta\in{\mathbb{R}}$ such that
\[
\vert x\vert \leq Q^{1+\eta}, \qquad \max_{1\leq j\leq n} \vert \zeta^{j}x-y_{j}\vert\leq Q^{-\frac{1}{n}+\eta} 
\]
has (at least) $j$ linearly independent solutions $(x,y_{1},\ldots,y_{n})\in{\mathbb{Z}^{n+1}}$. 
The functions $\psi_{n,j}(Q)$ can be equivalently defined via a lattice point problem, see~\cite{ss}. They have the properties
\[
-1\leq \psi_{n,j}(Q)\leq \frac{1}{n}, \qquad \qquad Q>1,\quad 1\leq j\leq n+1.
\]
Let 
\[
\underline{\psi}_{n,j}=\liminf_{Q\to\infty} \psi_{n,j}(Q),
\qquad \overline{\psi}_{n,j}=\limsup_{Q\to\infty} \psi_{n,j}(Q).
\]
These values clearly all lie in the interval $[-1,1/n]$. From Dirichlet's Theorem it follows that
$\psi_{n,1}(Q)\leq 0$ for all $Q>1$ and hence $\overline{\psi}_{n,1}\leq 0$.
For our purposes even more important will be the functions $\psi_{n,j}^{\ast}(Q)$
from~\cite{ss}. For $1\leq j\leq n+1$ and a parameter $Q>1$, define the value $\psi_{n,j}^{\ast}(Q)$ 
as the minimum of $\eta\in{\mathbb{R}}$ such that
\[
\vert H(P)\vert \leq Q^{\frac{1}{n}+\eta}, \qquad 
\vert P(\zeta)\vert\leq Q^{-1+\eta} 
\]
has (at least) $j$ linearly independent solutions in polynomials 
$P\in\mathbb{Z}[T]$ of degree at most $n$. 
See~\cite{ss} for the connection of the functions $\psi_{n,j}^{\ast}$
to a related lattice point problem, similarly as for simultaneous approximation.
Again put
\[
\underline{\psi}_{n,j}^{\ast}=\liminf_{Q\to\infty} \psi_{n,j}^{\ast}(Q),
\qquad \overline{\psi}_{n,j}^{\ast}=\limsup_{Q\to\infty} \psi_{n,j}^{\ast}(Q).
\]
For transcendental $\zeta$ Schmidt and Summerer~\cite[(1.11)]{ssch} established the inequalities 
\[
j\underline{\psi}_{n,j}+(n+1-j)\overline{\psi}_{n,n+1}\geq 0, 
\qquad j\overline{\psi}_{n,j}+(n+1-j)\underline{\psi}_{n,n+1}\geq 0,
\]
for $1\leq j\leq n+1$. The dual inequalities
\begin{equation} \label{eq:duales}
j\underline{\psi}_{n,j}^{\ast}+(n+1-j)\overline{\psi}_{n,n+1}^{\ast}\geq 0, 
\qquad j\overline{\psi}_{n,j}^{\ast}+(n+1-j)\underline{\psi}_{n,n+1}^{\ast}\geq 0,
\end{equation}
hold as well for the same reason.
As pointed out in~\cite{ss} Mahler's inequality implies
\begin{equation} \label{eq:mahler}
\vert\psi_{n,j}(Q)+\psi_{n,n+2-j}^{\ast}(Q)\vert \ll \frac{1}{\log Q}, \qquad 1\leq j\leq n+1. 
\end{equation}
In particular we have
\begin{equation} \label{eq:jaja}
\underline{\psi}_{n,j}=-\overline{\psi}_{n,n+2-j}^{\ast}, \qquad 
\overline{\psi}_{n,j}=-\underline{\psi}_{n,n+2-j}^{\ast},
\qquad 1\leq j\leq n+1.
\end{equation}
In particular all values $\underline{\psi}_{n,j}^{\ast}, \overline{\psi}_{n,j}^{\ast}$
lie in the interval $[-\frac{1}{n},1]$, and $\overline{\psi}_{n,1}^{\ast}\leq 0$ follows again
from Dirichlet's Theorem.
The constants $\underline{\psi}_{n,j}, \overline{\psi}_{n,j},\underline{\psi}_{n,j}^{\ast},\overline{\psi}_{n,j}^{\ast}$
relate to the classical approximation constants $\lambda_{n,j}=\lambda_{n,j}(\zeta),
w_{n,j}=w_{n,j}(\zeta)$ assigned to real $\zeta$ via
\begin{equation} \label{eq:umrechnen}
(1+\lambda_{n,j})(1+\underline{\psi}_{n,j})=
(1+\widehat{\lambda}_{n,j})(1+\overline{\psi}_{n,j})=\frac{n+1}{n}, \qquad 1\leq j\leq n+1,
\end{equation}
and 
\begin{equation} \label{eq:umrechnen2}
(1+w_{n,j})\Big(\frac{1}{n}+\underline{\psi}_{n,j}^{\ast}\Big)=
(1+\widehat{w}_{n,j})\Big(\frac{1}{n}+\overline{\psi}_{n,j}^{\ast}\Big)=\frac{n+1}{n}, \qquad 1\leq j\leq n+1.
\end{equation}
See~\cite[Theorem~1.4]{ss} for a proof of $j=1$ which can be readily extended to
the case of arbitrary $1\leq j\leq n+1$ as noticed in~\cite{j1}. 
From repeated application of \eqref{eq:jaja}, \eqref{eq:umrechnen} and \eqref{eq:umrechnen2} 
one can deduce
\begin{equation} \label{eq:debre}
\lambda_{n,j}(\zeta)= \frac{1}{\widehat{w}_{n,n+2-j}(\zeta)}, \qquad
\widehat{\lambda}_{n,j}(\zeta)= \frac{1}{w_{n,n+2-j}(\zeta)},
\end{equation}
for $1\leq j\leq n+1$, already noticed in~\cite{j2}. For $q>0$ we also define the functions
\begin{equation} \label{eq:lpsi}
L_{n,j}(q)= q\psi_{n,j}(Q), \qquad L_{n,j}^{\ast}(q)= q\psi_{n,j}^{\ast}(Q),
\end{equation}
where $Q=e^{q}$. They are piecewise linear with slopes among $\{-1,1/n\}$ and $\{-1/n,1\}$ respectively.
More precisely locally any $L_{n,j}$ coincides with some
\begin{equation} \label{eq:netto}
L_{\underline{x}}(q)=\max \left\{ \log \vert x\vert-q, 
\max_{1\leq j\leq n} \log \vert \zeta^{j}x-y_{j}\vert+\frac{q}{n}\right\}  
\end{equation}
where $\underline{x}=(x,y_{1},\ldots,y_{n})\in\mathbb{Z}^{n+1}$ for $y_{j}$ 
the closest integer to $\zeta^{j}x$, see \cite[page~75]{ss}.
Similarly any $L_{n,j}^{\ast}$ coincides locally with
\begin{equation} \label{eq:incide}
L_{P}^{\ast}(q)=\max \left\{ \log H(P)-\frac{q}{n}, 
\log \vert P(\zeta)\vert+q \right\}  
\end{equation}
for some $P\in\mathbb{Z}[T]$ of degree at most $n$. Observe that for fixed $P$ the left 
expression in \eqref{eq:incide} decays 
with slope $-1/n$ whereas the right expression rises with slope $1$ in the parameter $q$.
Consequently, at a local maximum of some $L_{n,j}^{\ast}$, the rising right expression of some $L_{P}^{\ast}(q)$ meets the falling left expression of some
$L_{Q}^{\ast}(q)$ with $H(Q)>H(P)$, and similarly for local 
maxima of $L_{n,j}$. On the other hand, at any local minimum $q$
of some $L_{n,j}^{\ast}$ there is either
equality in the expressions in \eqref{eq:incide} for some $P$, or the rising 
phase of some $L_{P}^{\ast}$ meets the falling phase of some $L_{Q}^{\ast}$
for some $Q$ with $H(Q)>H(P)$. In the first case, which always applies for $j=1$,
the function $L_{n,j}^{\ast}$ coincides with $L_{P}^{\ast}$ in a neighborhood of $q$.
The situation is again very similar for $L_{n,j}$. 
The identity \eqref{eq:umrechnen2}
has a parametric version in the sense that for any $(Q,\psi_{n,j}^{\ast}(Q))$
in the graph of some function $\psi_{n,j}^{\ast}$ then there exist $j$ linearly
independent polynomials $P_{1},\ldots,P_{j}\in\mathbb{Z}[T]$ 
of degree at most $n$ such that
\begin{equation}\label{eq:umrechnen3}
(1+w_{n}^{(j)})\left(\frac{1}{n}+\psi_{n,j}^{\ast}(Q)\right)= \frac{n+1}{n}+o(1), \qquad Q\to\infty,
\end{equation}
holds where
\[
w_{n}^{(j)}:= \frac{\min_{1\leq i\leq j} (-\log \vert P_{i}(\zeta)\vert)}{\max_{1\leq i\leq j}\log H(P_{i})},
\]
and vice versa. Very
similarly a dual parametric version of \eqref{eq:umrechnen} for the functions $\psi_{n,j}(Q)$
can be obtained. Both versions are basically inherited from the proof of~\cite[Theorem~1.4]{ss}.
A crucial observation for the parametric geometry of numbers developed in~\cite{ss},~\cite{ssch}
is that Minkowski's second lattice point Theorem translates into
\begin{equation} \label{eq:lsumme}
\left\vert \sum_{j=1}^{n+1} L_{n,j}(q)\right\vert \ll 1, \qquad 
\left\vert \sum_{j=1}^{n+1} L_{n,j}^{\ast}(q)\right\vert \ll 1.
\end{equation}
This implies that in any interval $I=(q_{1},q_{2})$ the sum of the differences
$L_{n,j}(q_{2})-L_{n,j}(q_{1})$ and $L_{n,j}^{\ast}(q_{2})-L_{n,j}^{\ast}(q_{1})$ 
over $1\leq j\leq n+1$
are bounded in absolute value as well by a fixed constant independent of $I$. 
We will implicitly use this fact in the proof of Theorem~\ref{ndrei}. This argument
is widely used in~\cite{ssch}.

\subsection{Two technical lemmata}

For the conditioned result \eqref{eq:uniformiert} we need (parts of) Lemma~\ref{deckel}
which is of some interest on its own. 
For its proof we will use that every local maximum of $L_{n,1}$ is 
a local minimum of $L_{n,2}$ (note: the analogue is in general false 
for $L_{n,j}, L_{n,j+1}$ when $j>1$). 
This follows from the elementary fact that for 
any vector $\underline{x}=(x,y_{1},\ldots,y_{n})\in\mathbb{Z}^{n+1}$ 
clearly any integral multiple $N\underline{x}$ cannot lead to a smaller value in \eqref{eq:lambda}.
Hence if two functions $L_{\underline{x}_{1}}, L_{\underline{x}_{2}}$ as in \eqref{eq:netto} 
induce two (successive) falling slopes $-1$ of $L_{n,1}$, with some rising phase of $L_{n,1}$ 
of slope $1/n$ in between, 
then the corresponding vectors $\underline{x}_{1}, \underline{x}_{2}$ are linearly independent, 
and the claim follows. Moreover we use $L_{n,1}(q)<0$ for all $q>0$, which is equivalent to Dirichlet's Theorem.

\begin{lemma} \label{deckel}
Let $n\geq 1$ be an integer and $\zeta$ be a real transcendental number. Assume
there is equality in either inequality of \eqref{eq:khintchine}, that is either
\begin{equation} \label{eq:rechts}
n\lambda_{n}(\zeta)+n-1=w_{n}(\zeta)
\end{equation}
or 
\begin{equation} \label{eq:links}
\lambda_{n}(\zeta)=\frac{w_{n}(\zeta)}{(n-1)w_{n}(\zeta)+n}
\end{equation}
holds. Then $\widehat{\lambda}_{n}(\zeta)=1/n$ and $\widehat{w}_{n}(\zeta)=n$.
\end{lemma}

\begin{proof}
Assume there is equality in the right inequality, that is $n\lambda_{n}(\zeta)+n-1=w_{n}(\zeta)$.
In case of $\lambda_{n}(\zeta)=\infty$ we have $\widehat{\lambda}_{n}(\zeta)=1/n$ and $\widehat{w}_{n}(\zeta)=n$
anyway by~\cite[Theorem~1.12 and Theorem~5.1]{schlei}.                              
Hence we can assume $\lambda_{n}(\zeta)<\infty$ which will simplify the estimates.
It suffices to show $\widehat{\lambda}_{n}(\zeta)=1/n$ since the two claims are well-known 
to be equivalent, which follows for example from \eqref{eq:ogerman}. 
It was shown by Schmidt and Summerer in the remark on page 80 below the proof of Theorem~1.4 
in~\cite{ss} that the right inequality in \eqref{eq:khintchine} is equivalent
to $\underline{\psi}_{n,1}+n \overline{\psi}_{n,n+1}\geq 0$. It follows directly from their 
deduction of the mentioned remark
that more generally the identity \eqref{eq:rechts} implies that for any $\varepsilon>0$
there exist arbitrarily large parameters $Q$ such that
\[
\vert\psi_{n,1}(Q)+n\psi_{n,j}(Q)\vert < \varepsilon, \qquad 2\leq j\leq n+1,
\]
where $Q$ can be chosen so that simultaneously $\psi_{n,1}(Q)$ is arbitrarily close to $\underline{\psi}_{n,1}$
and $\psi_{n,j}(Q)$ is arbitrarily close to $\overline{\psi}_{n,j}$ for $2\leq j\leq n+1$.
In particular, the identity \eqref{eq:rechts} implies 
\begin{equation} \label{eq:charakter}
\underline{\psi}_{n,1}=-n \overline{\psi}_{n,2}=-n \overline{\psi}_{n,3}=\cdots =-n \overline{\psi}_{n,n+1}
\end{equation}
and that for any $\epsilon>0$ and the (arbitrarily large) parameters $Q$ as above the estimate
\begin{equation}  \label{eq:sterndi}
0<\psi_{n,n+1}(Q)-\psi_{n,2}(Q)<\epsilon
\end{equation}
is satisfied. Moreover, since $\psi_{n,1}(Q)$ is close to $\underline{\psi}_{n,1}$,
we may assume that at such $Q$ the function $\psi_{n,1}$
has a local minimum, or equivalently $L_{n,1}$ has a local minimum at $\log Q$
(otherwise we get a contradiction to the definition of $\underline{\psi}_{1}$
either for some $\tilde{Q}<Q$ or some $\tilde{Q}>Q$
dependent on whether $\psi_{n,1}$ rises in some interval $(Q-\delta,Q)$ or decays in some interval $(Q,Q+\delta)$). 
Let $\epsilon>0$ and $Q_{1}$ be any fixed large value as above that in particular
satisfies \eqref{eq:sterndi}. Further let $q_{1}=\log Q_{1}$.
The estimate \eqref{eq:sterndi} can be written in terms of the functions $L_{n,.}$ in the way
\begin{equation} \label{eq:oben}
0<L_{n,n+1}(q_{1})-L_{n,2}(q_{1})<\epsilon\cdot q_{1}.
\end{equation}
From \eqref{eq:lsumme} we know that $L_{n,1}(q_{1})$ approximately equals $-\sum_{j=2}^{n+1} L_{n,j}(q_{1})$ up to addition of some constant, that is
\[
\left\vert L_{n,1}(q_{1})+\sum_{j=2}^{n+1} L_{n,j}(q_{1})\right\vert \leq C.
\] 
Since all $L_{n,2}(q_{1}),\ldots,L_{n,n+1}(q_{1})$ are roughly equal 
by \eqref{eq:oben}, we further deduce
\[
\vert L_{n,1}(q_{1})+nL_{n,2}(q_{1})\vert= 
\left\vert \left(L_{n,1}(q_{1})+\sum_{j=2}^{n+1} L_{n,j}(q_{1})\right)
+\sum_{j=2}^{n+1}(L_{n,2}(q_{1})-L_{n,j}(q_{1}))\right\vert
\leq C+n\epsilon q_{1},
\]
and hence in particular
\begin{equation} \label{eq:sieg}
L_{n,2}(q_{1})\geq -\frac{L_{n,1}(q_{1})}{n}-\epsilon q_{1}-\tilde{C}
\end{equation}
where $\tilde{C}=C/n$ is another constant. Now let $q_{0}$ be the largest value smaller than $q_{1}$
at which the function $L_{n,1}(q)$ has a local maximum.
Then by the assumption that $q_{1}$ is a local minimum of $L_{n,1}$ justified above, 
the function $L_{n,1}$ decays in the interval $[q_{0},q_{1}]$ with slope $-1$ such that
\begin{equation} \label{eq:haile}
L_{n,1}(q_{1})-L_{n,1}(q_{0})= q_{0}-q_{1}.
\end{equation}
On the other hand
\begin{equation} \label{eq:selasje}
L_{n,2}(q_{1})-L_{n,2}(q_{0})\leq \frac{q_{1}-q_{0}}{n}
\end{equation}
since the function $L_{n,2}(q)$ has slope at most $1/n$. Moreover, since any local 
maximum of $L_{n,1}(q)$ is a local minimum of $L_{n,2}(q)$, we have
\[
L_{n,1}(q_{0})=L_{n,2}(q_{0}).
\]
Combination with \eqref{eq:haile} and \eqref{eq:selasje} yields
\[
L_{n,2}(q_{1})-L_{n,1}(q_{1})\leq \Big(1+\frac{1}{n}\Big)(q_{1}-q_{0}).
\]
Together with \eqref{eq:sieg} we obtain
\[
L_{n,1}(q_{1}) \geq L_{n,2}(q_{1})-\Big(1+\frac{1}{n}\Big)(q_{1}-q_{0})
\geq -\frac{L_{n,1}(q_{1})}{n}-\epsilon q_{1}-\tilde{C}-\Big(1+\frac{1}{n}\Big)(q_{1}-q_{0})
\]
which yields
\[
L_{n,1}(q_{1}) \geq -\frac{n\epsilon}{n+1}q_{1}-\tilde{C}-(q_{1}-q_{0}).
\]
Together with \eqref{eq:haile} we infer
\[
L_{n,1}(q_{0}) \geq -\frac{n\epsilon}{n+1}q_{1}-\tilde{C}.
\]
Now the assumption $\lambda_{n}(\zeta)<\infty$ implies with \eqref{eq:umrechnen}
that $\underline{\psi}_{n,1}>-1$ and from this it is not hard to see that $q_{1}\ll q_{0}$
for all $q_{0},q_{1}$ as above with a constant depending only on $\lambda_{n}(\zeta)$
or equivalently $\underline{\psi}_{n,1}$. Hence, for $q_{0}>1$, we have
\[
0>L_{n,1}(q_{0}) \gg -\epsilon q_{0}. 
\]
Since by the transcendence of $\zeta$ the values $q_{0}$ induced from $q_{1}$ as above
clearly tend to infinity as $q_{1}$ does, we infer $\overline{\psi}_{n,1}=0$ as we may choose
$\epsilon$ arbitrarily small. By \eqref{eq:umrechnen} this is again
equivalent to $\widehat{\lambda}_{n}(\zeta)=1/n$. The proof in case of equality in the right inequality is finished. 

We only sketch the deduction of the dual result. Assume the identity \eqref{eq:links} holds. 
The dual characterization $\underline{\psi}_{n,1}^{\ast}+n \overline{\psi}_{n,n+1}^{\ast}\geq 0$
from~\cite{ss} for the related left inequality in \eqref{eq:khintchine} yields 
the dual characterization for the equality \eqref{eq:links} for the same reasons.
Proceeding as above yields very similarly as above
$0<\psi_{n,n+1}^{\ast}(Q)-\psi_{n,2}^{\ast}(Q)<\epsilon$ 
for large $Q$ for which $\log Q$ are local minima of $L_{n,1}^{\ast}$ and such
that $\psi_{n,1}^{\ast}(Q)$ is close to $\underline{\psi}_{n,1}^{\ast}$, dual to \eqref{eq:sterndi}.
For such $Q$ we now look at the smallest local maximum of $L_{n,1}^{\ast}$
greater than $\log Q$. Since all $L_{n,j}^{\ast}$ have slope within $\{-1/n,1\}$,
the claim $\widehat{w}_{n}(\zeta)=n$ follows very similarly incorporating that
any local maximum of $L_{n,1}^{\ast}$ is a local minimum of $L_{n,2}^{\ast}$ again.
\end{proof}

\begin{remark} \label{kuh}
We point out that the proof of Lemma~\ref{deckel} does not require that the point lies on the Veronese curve 
defined as $\{(t,t^{2},\ldots,t^{k}): t\in{\mathbb{R}}\}$. The only point where we used the special
form of successive powers was for $\lambda_{n}(\zeta)=\infty$, and in this case more concise estimates
show the claim as well. Hence the claim extends naturally to the analogue exponents
assigned to $\underline{\zeta}\in\mathbb{R}^{k}$ whose coordinates
are linearly independent together with $\{1\}$.
\end{remark}

It will be convenient to utilize the following Lemma~\ref{technisch} for the proof
of Theorem~\ref{ndrei}. Roughly 
speaking, it shows that multiplication of a polynomial $P$ with a polynomial $Q$ for which 
$\vert Q(\zeta)\vert\approx H(Q)^{-1}$ holds, induces an increase of the corresponding
function $L_{3,.}^{\ast}$ by $1/3$ in some interval. 
For fixed real $\zeta$ we will say a polynomial $P\in\mathbb{Z}[T]$ of degree at most $3$ 
{\em induces a point $(q,L_{P}^{\ast}(q))$ in the $3$-dimensional Schmidt-Summerer diagram} 
if $(q,L_{P}^{\ast}(q))$ is the local minimum of $L_{P}^{\ast}$ implicitly defined 
via $H(P), P(\zeta)$ by 
\begin{equation} \label{eq:diagramm}
L_{P}^{\ast}(q)=\log H(P)-\frac{q}{3}=\log \vert P(\zeta)\vert+q,
\end{equation}
consistent with \eqref{eq:incide}. 
Recall that any local minimum of some successive minimum function $L_{3,.}^{\ast}$ is obtained
as in \eqref{eq:diagramm} for some $P\in\mathbb{Z}[T]$.
\begin{lemma} \label{technisch}
Let $P,Q,R\in\mathbb{Z}[T]$ be of large heights and such that $R=PQ$ and $R$ has degree at most three. 
Assume $P$ induces the point $(q_{1},L_{P}^{\ast}(q_{1}))$
and $R$ induces the point $(q_{2},L_{R}^{\ast}(q_{2}))$
in the $3$-dimensional Schmidt-Summerer diagram. Further assume 
\begin{equation} \label{eq:nicht}
\vert Q(\zeta)\vert= H(Q)^{-1+\delta}
\end{equation}   
for $\delta$ of small absolute value, and that $(\log H(Q))^{-1}=O(\delta)$. Then 
\begin{equation} \label{eq:convers}
\frac{L_{R}^{\ast}(q_{2})-L_{P}^{\ast}(q_{1})}{q_{2}-q_{1}}=\frac{1}{3}+O(\delta).
\end{equation}

\end{lemma}
\begin{proof}
From \eqref{eq:diagramm} we calculate
\[
q_{1}= \frac{3}{4}\cdot (\log H(P)-\log \vert P(\zeta)\vert), \qquad 
L_{P}^{\ast}(q_{1})=\frac{3}{4}\cdot \log H(P) + \frac{1}{4}\cdot \log \vert P(\zeta)\vert.
\]
Similarly, we infer
\[
q_{2}= \frac{3}{4}\cdot (\log H(R)-\log \vert R(\zeta)\vert)=
\frac{3}{4}\cdot (\log H(P)+\log H(Q)+\Delta-(\log \vert P(\zeta)\vert+\log \vert Q(\zeta)\vert)),
\]
and 
\[
L_{R}^{\ast}(q_{2})=\frac{3}{4}\cdot (\log H(P)+\log H(Q)+\Delta)+
\frac{1}{4}\cdot (\log \vert P(\zeta)\vert+\log \vert Q(\zeta)\vert)
\]
where $\Delta$ is bounded by virtue of \eqref{eq:multipli} below. Inserting yields
\[
\frac{L_{R}^{\ast}(q_{2})-L_{P}^{\ast}(q_{1})}{q_{2}-q_{1}}=
\frac{\frac{3}{4} \log H(Q)+\frac{1}{4} \log \vert Q(\zeta)\vert+ \frac{3}{4}\Delta}
{\frac{3}{4} \log H(Q)-\frac{3}{4} \log \vert Q(\zeta)\vert+ \frac{3}{4}\Delta},
\]
and with the assumption \eqref{eq:nicht} further
\[
\frac{L_{R}^{\ast}(q_{2})-L_{P}^{\ast}(q_{1})}{q_{2}-q_{1}}=
\frac{\left(\frac{1}{2}+\frac{1}{4}\delta\right)\log H(Q)+ \frac{3}{4}\Delta}
{\left(\frac{3}{2}-\frac{3}{4}\delta\right)\log H(Q)+ \frac{3}{4}\Delta}.
\]
The claim follows by elementary rearrangements using the assumption $(\log H(Q))^{-1}=O(\delta)$.
\end{proof}

Conversely \eqref{eq:convers} implies that $\log \vert Q(\zeta)\vert/\log H(Q)+1$ is small
by a very similar argument, but we will not use this. Again the proposition did not use
the fact that we deal with successive powers of a number, and can be generalized to any dimension. 

\section{Proofs of Theorems~\ref{ndrei},~\ref{sternchen} and \ref{nvier}} \label{proofs}

Apart from Theorem~\ref{roythm} and the concepts of Section~\ref{para},
we will use that
for any polynomials $Q_{1},Q_{2}$ with integral coefficients of degree bounded by $n$ we have
\begin{equation} \label{eq:multipli}
H(Q_{1}Q_{2})\asymp_{n} H(Q_{1})H(Q_{2}).
\end{equation}
See~\cite[Hilfssatz~3]{wirsing}. As in our applications the dimensions $n$ are fixed
we can assume absolute constants in \eqref{eq:multipli}.
We will sometimes implicitly use the consequence 
that if $Q=Q_{1}Q_{2}$ then $\vert Q(\zeta)\vert \leq H(Q)^{-z}$ 
implies that either $\vert Q_{1}(\zeta)\vert \ll H(Q_{1})^{-z}$
or $\vert Q_{2}(\zeta)\vert \ll H(Q_{2})^{-z}$ must be satisfied, which
was essentially used by Wirsing~\cite{wirsing}.
We start with the proof of Theorem~\ref{nvier} since it is the least technical one.

\begin{proof}[Proof of Theorem~\ref{nvier}]
We will prove that any extremal number $\zeta$ satisfies
\begin{equation}  \label{eq:vierte}
w_{4,4}(\zeta)\geq \rho.
\end{equation}
Assume we have already shown \eqref{eq:vierte}. Then the 
unconditional claim \eqref{eq:lammda} follows from iterated use
of results from parametric geometry of numbers. Indeed, from \eqref{eq:vierte} applying \eqref{eq:umrechnen2} with $n=j=4$ we first obtain
\begin{equation}  \label{eq:armst}
\underline{\psi}_{4,4}^{\ast}\leq \frac{2-\sqrt{5}}{4(3+\sqrt{5})}.
\end{equation}
In view of \eqref{eq:jaja} and \eqref{eq:duales} applied with $n=j=4$, we obtain 
\begin{equation} \label{eq:recal}
\underline{\psi}_{4,1}=-\overline{\psi}_{4,5}^{\ast} \leq 4\cdot \underline{\psi}_{4,4}^{\ast} 
\leq \frac{2-\sqrt{5}}{3+\sqrt{5}}. 
\end{equation}
Eventually computing the corresponding value of $\lambda_{4}$
by applying \eqref{eq:umrechnen} with $n=4, j=1$ leads precisely to the lower bound $\gamma/2$ in the theorem.

We are left to prove \eqref{eq:vierte}. For this we use the characterization of the 
polynomials $P_{k}\in{\mathbb{Z}[T]}$ of degree $2$ for $n=2$ from Theorem~\ref{roythm}. 
Consider for fixed large $k$ three successive polynomials $P_{k-2},P_{k-1},P_{k}$. 
Then we know from Theorem~\ref{roythm} that 
\begin{equation} \label{eq:property}
\vert P_{j}(\zeta)\vert\asymp H(P_{j})^{-\rho}, \qquad j\in\{k-2,k-1,k\}.
\end{equation}
Applied with $j=k$ it is obvious that
the polynomials $R_{k}(T)=TP_{k}(T), S_{k}(T)=T^{2}P_{k}(T)$ have degrees $3$ and $4$,
heights $H(P_{k})=H(R_{k})=H(S_{k})$, and satisfy
\[
\vert P_{k}(\zeta)\vert\asymp_{\zeta} \vert R_{k}(\zeta)\vert\asymp_{\zeta} \vert S_{k}(\zeta)\vert 
\asymp_{\zeta} H(P_{k})^{-\rho}
\]
as well. The polynomials $P_{k},R_{k},S_{k}$ are obviously linearly independent 
and hence $w_{4,3}(\zeta)\geq \rho$.
As the fourth polynomial $T_{k}$ we take the product of $P_{k-1}$ and $P_{k-2}$.
First we show that $\{ P_{k},R_{k},S_{k},T_{k}\}$ are linearly independent. Otherwise $T_{k}=P_{k-1}P_{k-2}$
would lie in the $3$-dimensional space spanned by $P_{k},R_{k},S_{k}$, which by the special form
of $R_{k},S_{k}$ means $T_{k}=P_{k}Z$ for some polynomial $Z(T)\in{\mathbb{Q}[T]}$ of degree $2$.
However we know from Theorem~\ref{roythm} that the best approximating polynomials $P_{j}$ 
are irreducible over $\mathbb{Z}[T]$ for all large $j$. 
Hence by the unique factorization in $\mathbb{Z}[T]$ the polynomial
$P_{k}$ must equal (up to sign)
either $P_{k-1}$ or $P_{k-2}$, which is clearly false, contradiction.

Moreover from \eqref{eq:property} and the characterization in Theorem~\ref{roythm} it 
is known that $H(P_{k-2})^{\nu^{2}}\asymp H(P_{k-1})^{\nu}\asymp H(P_{k})$.
Since $\nu^{-1}+\nu^{-2}=1$ and $H(T_{k})\asymp H(P_{k-1})H(P_{k-2})$ by \eqref{eq:multipli},
we deduce $H(T_{k}) \asymp H(P_{k})$. Together with property \eqref{eq:property} for 
$j=k-1$ and $j=k-2$ we infer
\[
\vert T_{k}(\zeta)\vert=\vert P_{k-1}(\zeta)P_{k-2}(\zeta)\vert \asymp_{\zeta} H(P_{k-1})^{-\rho}H(P_{k-2})^{-\rho}
\asymp H(P_{k-1}P_{k-2})^{-\rho}\asymp H(P_{k})^{-\rho}.
\]
Summing up, we have found four linearly independent polynomials $P_{k},R_{k},S_{k},T_{k}$ 
with the properties
\[
H(P_{k})\asymp H(R_{k})\asymp H(S_{k})\asymp H(T_{k})
\]
and
\[
\vert P_{k}(\zeta)\vert \asymp_{\zeta} \vert R_{k}(\zeta)\vert \asymp_{\zeta} \vert S_{k}(\zeta)\vert 
\asymp_{\zeta} \vert T_{k}(\zeta)\vert \asymp_{\zeta} H(P_{k})^{-\rho}.
\]
Since this holds for any large $k$ we have established \eqref{eq:vierte}. 

Finally we show the conditioned results. The
equality $\lambda_{4}(\zeta)=\gamma/2$ follows immediately from Khintchine's inequalities \eqref{eq:khintchine}
since the upper bound for $\lambda_{4}(\zeta)$ that arises from
$n=4, w_{4}(\zeta)=\rho$, coincides with the lower bound $\gamma/2$
established above (the argument essentially used the characterization \eqref{eq:charakter}, \eqref{eq:sterndi}
for equality \eqref{eq:rechts} from~\cite{ss} used in the proof of Proposition~\ref{deckel}). 
Finally \eqref{eq:uniformiert} follows from Lemma~\ref{deckel}
since we have just shown that $w_{4}(\zeta)=\rho$ implies the identity \eqref{eq:rechts} for 
any extremal number $\zeta$ and $n=4$.
\end{proof}

\begin{remark} \label{hirsch}
It was essentially shown in the proof of~\cite[Theorem~2]{bug} that the condition
\begin{equation} \label{eq:bugbed}
w_{1}(\zeta)=w_{2}(\zeta)=\cdots=w_{n}(\zeta)
\end{equation}
implies \eqref{eq:rechts}. 
If the hypothesis $w_{4}(\zeta)=\rho$ of Theorem~\ref{nvier} holds
then its assertion and \eqref{eq:werte} show that extremal numbers 
provide counterexamples for the reverse implication
for $n=4$. In this context note that if $\lambda_{n}(\zeta)>1$ the claims \eqref{eq:bugbed} and \eqref{eq:rechts}
are indeed equivalent by~\cite[Theorem~5.4]{schlei}.
Note also that from Lemma~\ref{deckel} and the above implication we could deduce that \eqref{eq:bugbed}
implies $\widehat{\lambda}_{n}(\zeta)=1/n$ and $\widehat{w}_{n}(\zeta)=n$. 
However, the weaker condition $w_{1}(\zeta)\geq n$ already implies $\widehat{\lambda}_{n}(\zeta)=1/n$ 
and $\widehat{w}_{n}(\zeta)=n$ as established in~\cite[Theorem~5.1]{schlei}.
\end{remark}

The proof of Theorem~\ref{nvier} in fact provides upper bounds for the frequency of good
simultaneous rational approximations to $(\zeta,\zeta^{2},\zeta^{3},\zeta^{4})$. 
More precisely the proof shows that there exists a sequence $(x_{k})_{k\geq 1}$ of 
positive integers that satisfy
\[
x_{k+1}\ll x_{k}^{\nu}, \qquad
\max_{1\leq j\leq 4} \Vert x_{k}\zeta^{j}\Vert\ll x_{k}^{-\gamma/2}.
\]
In case of the conjectured equality in \eqref{eq:lammda} we even have
 \begin{equation} \label{eq:dennoch}
 x_{k+1}\asymp x_{k}^{\nu}, \qquad
 \max_{1\leq j\leq 4} \Vert x_{k}\zeta^{j}\Vert\asymp x_{k}^{-\gamma/2}.
 \end{equation}
Here as usual $\Vert.\Vert$ denotes the distance to the nearest integer. 
We briefly sketch how to deduce these facts from the proof above.
The polynomials $P_{k},R_{k},S_{k},T_{k}$ in the proof
which induce the bound for the value $\underline{\psi}_{4,4}^{\ast}$ in \eqref{eq:armst} 
appear with frequency $H(P_{k+1})\asymp H(P_{k})^{\nu}$ (and very
similarly for $R_{k}, S_{k}, T_{k}$). The last minimum $\psi_{4,5}^{\ast}(Q)$ 
at the corresponding positions $Q$ in the Schmidt-Summerer diagram is asymptotically
bounded below as in \eqref{eq:recal} and the corresponding polynomials appear with
the same logarithmic asymptotic height frequency $\nu$.
We now flip
the diagram along the horizontal axis according to \eqref{eq:mahler}
to obtain (roughly) the dual problem of simultaneous approximation. Thereby with simple geometric considerations involving \eqref{eq:netto} and reinterpreting
to classical exponents $\lambda_{4,.}$ we see that the first coordinates of best approximations related to the bound for $\underline{\psi}_{4,1}$
in \eqref{eq:recal} appear with frequency
$x_{k+1}\ll x_{k}^{\nu}$ as well (with a technical proof it possible to show 
that a single $x_{k}$ cannot induce the good approximations for two consecutive values 
of $Q$ obtained this way). In case of equality in \eqref{eq:lammda} the functions
$\psi_{4,1}(Q)$ must have a local minimum at such places $Q$ and \eqref{eq:dennoch} follows.
It is tempting to further
conjecture that for the corresponding approximation
vectors $(x_{k},y_{k,1},\ldots,y_{k,4})_{k\geq 1}$, 
where $x_{k}$ is as in \eqref{eq:dennoch} and $y_{k,j}$ is the closest integer to $\zeta^{j} x_{k}$,
similar general recursive patterns as for $n=2$ noticed in~\cite{royyy} exist.
However, we do not further investigate this topic here. 

We turn to the case $n=3$. For a real number $\zeta$ we define the sequence of $1$-dimensional
{\em best approximation polynomials} $(E_{l})_{l\geq 1}$ attached to $\zeta$. They are given by 
linear polynomials $E_{l}(T)=a_{l}T+b_{l}$ with $a_{l},b_{l}\in{\mathbb{Z}}$
defined by $E_{1}(T)=T-\lfloor \zeta\rfloor$ and $E_{l+1}$ is recursively defined via $E_{l}$
as the linear polynomial of least height for which $0<\vert E_{l+1}(\zeta)\vert<\vert E_{l}(\zeta)\vert$.
These polynomials obviously satisfy $H(E_{1})<H(E_{2})<\cdots$ and 
\[
E_{l}(\zeta)= \min \{ \vert Q(\zeta)\vert: Q\in\mathbb{Z}[T], \; \deg(Q)=1,\; 1\leq H(Q)\leq H(E_{l})\}.
\]
It follows from the theory of continued fractions that the rational numbers $b_{l}/a_{l}$ are
precisely the convergents to $\zeta$. 
Moreover by Dirichlet's Theorem the best approximating polynomials satisfy
\begin{equation} \label{eq:eeg}
\vert E_{l}(\zeta)\vert \ll_{\zeta} H(E_{l})^{-1}, \qquad l\geq 1.
\end{equation}
Furthermore it is well-known and follows from elementary results on the theory of continued fractions that
$\vert E_{l}(\zeta)\vert\asymp_{\zeta} H(E_{l+1})^{-1}$ for all irrational $\zeta$, which readily implies
\begin{equation} \label{eq:nurmehr}
1\leq \liminf_{l\to\infty} \frac{\log H(E_{l+1})}{\log H(E_{l})} \leq
\limsup_{l\to\infty} \frac{\log H(E_{l+1})}{\log H(E_{l})}= \lambda_{1}(\zeta).
\end{equation}

In view of the rather technical proof of
\eqref{eq:zweite}, for the convenience of the reader
we give a brief outline of some facts we will show in the course of the proof.
We will establish a rather precise description of the functions
$L_{3,1}^{\ast}(q),\ldots,L_{3,4}^{\ast}(q)$ on $q\in(0,\infty)$ 
induced by an extremal number, its square and its cube.
Denote by $\vert I\vert$ the length of an interval $I$. We will show
there exists a partition of the positive real numbers 
in successive intervals $I_{1},J_{1},I_{2},J_{2},\ldots$ with the following properties.
  
\begin{itemize}
\item $\lim_{k\to\infty} \vert I_{k}\vert/\vert J_{k}\vert=1$
\item $\lim_{k\to\infty} \vert I_{k+1}\vert/\vert I_{k}\vert=\lim_{k\to\infty} \vert J_{k+1}\vert/\vert J_{k}\vert=\nu$.
\item At the beginning of every $I_{k}$ all $L_{3,i}^{\ast}(q)$ are all small (more precisely
$o(q)$ as $q\to\infty$) by absolute value. Then in $I_{k}$ the
functions $L_{3,1}^{\ast}(q), L_{3,2}^{\ast}(q)$ basically decay with slope $-1/3$,
whereas $L_{3,3}^{\ast}(q),L_{3,4}^{\ast}(q)$ basically rise with slope $1/3$ in any not too short
subinterval of $I_{k}$ (clearly not in too short intervals, since $L_{3,.}^{\ast}$ have slope within $\{-1/3,1\}$). 
\item At the end of $I_{k}$ and 
beginning of $J_{k}$ the opposite behavior appears, that is $L_{3,1}^{\ast}(q), L_{3,2}^{\ast}(q)$ basically
rise with slope $1/3$ on any not too short subinterval of $J_{k}$,
whereas $L_{3,3}^{\ast}(q), L_{3,4}^{\ast}(q)$ basically decay with slope $-1/3$
until the functions $L_{3,1}^{\ast},\ldots,L_{3,4}^{\ast}$ asymptotically meet again at the end of $J_{k}$ 
which is the beginning of $I_{k+1}$. 
\item The functions $\vert L_{3,1}^{\ast}(q)-L_{3,2}^{\ast}(q)\vert$ 
such as $\vert L_{3,3}^{\ast}(q)-L_{3,4}^{\ast}(q)\vert$ are bounded uniformly in $q$.
\end{itemize}

All above is basically true for the simultaneous approximation functions $L_{3,j}(q)$ as well 
by \eqref{eq:jaja}. Observe that by the last point above in particular
\begin{align}
w_{3,1}(\zeta)&=w_{3,2}(\zeta), \quad w_{3,3}(\zeta)= w_{3,4}(\zeta), \quad
\widehat{w}_{3,1}(\zeta)=\widehat{w}_{3,2}(\zeta), 
\quad \widehat{w}_{3,3}(\zeta)=\widehat{w}_{3,4}(\zeta),     \label{eq:rumpel} \\
\lambda_{3,1}(\zeta)&=\lambda_{3,2}(\zeta), \quad \lambda_{3,3}(\zeta)= \lambda_{3,4}(\zeta), \quad
\widehat{\lambda}_{3,1}(\zeta)=\widehat{\lambda}_{3,2}(\zeta), 
\quad \widehat{\lambda}_{3,3}(\zeta)=\widehat{\lambda}_{3,4}(\zeta), \label{eq:stielzchen}
\end{align}
which extends the claim of Theorem~\ref{ndrei}. See also Remark~\ref{dasreh} below.
We point out that roughly speaking the decay phases of $L_{3,.}^{\ast}$ are induced by the polynomials
$P_{k}$ from Theorem~\ref{roythm}. The rising phases are induced by products $P_{k}E_{l}$
for fixed $P_{k}$ and suitable successive best approximating polynomials $E_{l}$ defined above, 
which indeed lead to asymptotic increase by $1/3$ as stated in the description above,
basically in view of Lemma~\ref{technisch}.

\begin{proof}[Proof of Theorem~\ref{ndrei}]
First we prove \eqref{eq:erste}. We show that
\begin{equation} \label{eq:34}
w_{3,4}(\zeta)\geq 3.
\end{equation}
Provided this is true it follows immediately that $w_{4,3}(\zeta)=\widehat{w}_{3}(\zeta)=3$, since 
$w_{3,4}(\zeta)=\widehat{\lambda}_{3}(\zeta)^{-1}\leq 3$ by \eqref{eq:debre} and \eqref{eq:ldiri}.
This argument in fact utilizes parametric geometry of numbers. 
Eventually it is well-known and follows for example from \eqref{eq:ogerman}
that both claims in \eqref{eq:erste} are equivalent.

For \eqref{eq:erste} it remains to be shown that \eqref{eq:34} holds. 
Let $k$ be fixed large and consider the polynomials $P_{k},P_{k+1},\ldots$ from
Theorem~\ref{roythm}, and let $R_{j}(T)=TP_{j}(T)$ for $j\geq k$.
Further let $X=H(P_{k+1})$. Then obviously $P_{k+1}(T)$ and $R_{k+1}(T)=TP_{k+1}$ satisfy
\begin{equation} \label{eq:fisch}
H(P_{k+1})=H(R_{k+1})=X, \qquad \vert P_{k+1}(\zeta)\vert \asymp_{\zeta} 
\vert R_{k+1}(\zeta)\vert \asymp_{\zeta} X^{-\rho}<X^{-3}.
\end{equation}
Let $\epsilon>0$. We shall construct polynomial multiples
\begin{equation} \label{eq:kulti}
Q_{k,1}=R_{k,1}\cdot P_{k}, \qquad Q_{k,2}=R_{k,2}\cdot P_{k}
\end{equation}
of $P_{k}$ with $R_{k,i}\in{\mathbb{Z}[T]}$ polynomials of 
degree one such that $\{ R_{k,1},R_{k,2}\}$ 
and hence also $\{ Q_{k,1},Q_{k,2}\}$ are linearly independent and satisfy
\begin{equation} \label{eq:thun}
H(Q_{k,i})\ll X, \qquad \vert Q_{k,i}(\zeta)\vert \ll X^{-3+\epsilon}, \qquad i\in\{1,2\}.
\end{equation}
One readily verifies that $\{ Q_{k,1},Q_{k,2}\}$ span the same space as $\{ P_{k},TP_{k}\}$ indifferent
which linear polynomials $R_{k,i}$ we choose. 
Observe that the space spanned by $\{ P_{k+1},R_{k+1},Q_{k,1},Q_{k,2}\}$ consequently
has dimension $4$. Indeed otherwise the polynomial identity $P_{k}(T)Y_{1}(T)=P_{k+1}(T)Y_{2}(T)$
would have linear integer polynomial solutions $Y_{1},Y_{2}$, contradiction since $P_{k},P_{k+1}$
have degree two and are irreducible and not proportional
and $\mathbb{Z}[T]$ has unique factorization. 
Hence from \eqref{eq:fisch} and \eqref{eq:thun} 
indeed the claim \eqref{eq:34} follows by considering $\{ P_{k+1},R_{k+1},Q_{k,1},Q_{k,2}\}$ 
as $\epsilon$ can be chosen arbitrarily small. To finally prove \eqref{eq:thun}, for 
the given $X=H(P_{k+1})$ we let
$R_{k,1}=E_{l}$ and $R_{k,2}=E_{l+1}$ be two successive best approximating polynomials 
in dimension $n=1$ as introduced before the proof
with $l$ chosen largest possible such that still $H(R_{k,i})H(P_{k})\leq X$ for $i\in\{1,2\}$. 
It follows from \eqref{eq:multipli} and \eqref{eq:kulti} that                        
\begin{equation} \label{eq:kurz}
H(Q_{k,i})\ll X, \qquad i\in\{1,2\}. 
\end{equation}
On the other hand, since extremal numbers satisfy $\lambda_{1}(\zeta)=1$ as mentioned in \eqref{eq:werte}, 
by \eqref{eq:nurmehr} the sequence $(E_{l}(T))_{l\geq 1}$ of 
best approximating polynomials in dimension $1$ satisfies 
\begin{equation} \label{eq:unibaer}
\lim_{l\to\infty} \frac{\log H(E_{l+1})}{\log H(E_{l})}=1, \qquad 
\lim_{l\to\infty} -\frac{\log \vert E_{l}(\zeta)\vert}{H(E_{l})}=1.
\end{equation}
Since $R_{k,1}=E_{l}, R_{k,2}=E_{l+1}$ and by our maximal choice of $l$, 
it is not hard to see that
\[
H(Q_{k,i})\geq X^{1-\epsilon}, \qquad i\in\{1,2\}.
\]
It further follows from \eqref{eq:multipli} and $H(P_{k+1})\asymp H(P_{k})^{\nu}$ 
or equivalently $H(P_{k})\asymp H(P_{k+1})^{\gamma}$ in view of Theorem~\ref{roythm}, that we have
\[
H(R_{k,i})\gg H(Q_{k,i})H(P_{k})^{-1}\gg
X^{1-\epsilon}H(P_{k})^{-1}\gg X^{1-\gamma-\epsilon}, \qquad i\in\{1,2\}.
\]
Together with \eqref{eq:eeg} this leads to
\[
\vert R_{k,i}(\zeta)\vert \ll_{\zeta} X^{-1+\gamma+\epsilon}, \qquad i\in\{1,2\}. 
\]
Hence 
\[
\vert Q_{k,i}(\zeta)\vert= \vert P_{k}(\zeta)\vert\cdot \vert R_{k,i}(\zeta)\vert
\ll_{\zeta} X^{-\rho \gamma}\cdot X^{-1+\gamma+\epsilon}=X^{-3+\epsilon}, \qquad i\in\{1,2\},
\]
where we used $\rho\gamma+1-\gamma=3$, which can be readily checked. Thus recalling \eqref{eq:kurz}
we have proved \eqref{eq:thun} and hence together with \eqref{eq:fisch} finally \eqref{eq:erste}.

Now we prove the more technical identities \eqref{eq:zweite}. In the proof of \eqref{eq:erste} 
above we have shown that for any large $k$, with $X=H(P_{k+1})$ we have four linearly independent 
polynomials $\{ T_{1},\ldots,T_{4}\}=\{ P_{k+1},R_{k+1},Q_{k,1},Q_{k,2}\}$ 
with $H(T_{i})\ll X$ and $\vert T_{i}(\zeta)\vert \leq X^{-3+\epsilon}$. 
Following the proof of \eqref{eq:umrechnen2}, this means that for arbitrarily small $\varepsilon>0$,
any large $k$ induces $q_{k}>0$ such that all 
\begin{equation} \label{eq:virtus}
\vert L_{3,i}^{\ast}(q_{k})\vert\leq \varepsilon q_{k}, \qquad 1\leq i\leq 4, 
\end{equation}
where $\lim_{k\to\infty} q_{k}/\log H(P_{k+1})=3$ in view of \eqref{eq:incide}. 
Since by Theorem~\ref{roythm} any polynomial $P_{k+1}$ induces an approximation 
of quality
\[
-\frac{\log \vert P_{k+1}(\zeta)\vert}{\log H(P_{k+1})}= \rho+o(1)>3,\qquad k\to\infty,
\]
and so does $R_{k+1}(T)=TP_{k+1}(T)$, it follows that $L_{3,1}^{\ast}$ and $L_{3,2}^{\ast}$ 
decay with asymptotic slope $-1/3$ in some interval $(q_{k},b_{k})$ and $(q_{k},c_{k})$ respectively,
for $b_{k}$ and $c_{k}$ local minima of $L_{3,1}^{\ast}$ and $L_{3,2}^{\ast}$ respectively.
More precisely, the local minima $(d_{k},L_{P_{k+1}}^{\ast}(d_{k}))$ and $(e_{k},L_{R_{k+1}}^{\ast}(e_{k}))$
of the functions $L_{P_{k+1}}^{\ast}$ and $L_{R_{k+1}}^{\ast}$ as in \eqref{eq:incide}, respectively almost coincide
with local minima $(b_{k},L_{3,1}^{\ast}(b_{k}))$ and $(c_{k},L_{3,2}^{\ast}(c_{k}))$.
By this more precisely we mean that all differences 
\[
\vert b_{k}-d_{k}\vert, \quad \vert b_{k}-e_{k}\vert, \quad
\vert c_{k}-d_{k}\vert, \quad \vert c_{k}-e_{k}\vert
\]
as well as the corresponding differences of the $L^{\ast}$ evaluations
\begin{align*}
\vert L_{3,1}^{\ast}(b_{k})-L_{P_{k+1}}^{\ast}(d_{k})\vert, \qquad  
\vert L_{3,1}^{\ast}(b_{k})-L_{R_{k+1}}^{\ast}(e_{k})\vert,  \\
\vert L_{3,2}^{\ast}(c_{k})-L_{P_{k+1}}^{\ast}(d_{k})\vert, \qquad
\vert L_{3,2}^{\ast}(c_{k})-L_{R_{k+1}}^{\ast}(e_{k})\vert,
\end{align*}
at these points are bounded by a fixed constant for all $k$. Very similarly
it is obvious from the fact that $P_{k+1}(\zeta)$ and $R_{k+1}(\zeta)$ differ 
only by the factor $\zeta$ that $b_{k}$ and $c_{k}$ are asymptotically equal, 
by which we mean their ratio $b_{k}/c_{k}$ tends to one (in fact their difference 
$\vert b_{k}-c_{k}\vert$ is again bounded) as $k\to\infty$. 
Hence with the parametric formula \eqref{eq:umrechnen3} for the parameter $w_{3}^{(1)}=w_{3}^{(2)}=\rho$, with
\[
Q_{k}:=e^{b_{k}},  \qquad k\geq 1,
\]
(not to confuse with the polynomials $Q_{k,i}$) we calculate
\begin{equation} \label{eq:brauche}
\lim_{k\to\infty} \psi_{3,1}^{\ast}(Q_{k})= \lim_{k\to\infty} \psi_{3,2}^{\ast}(Q_{k})
=\frac{1-\sqrt{5}}{3(3+\sqrt{5})}.  
\end{equation}
Since $L_{3,1}^{\ast}$ and $L_{3,2}^{\ast}$ both decay with asymptotic slope $-1/3$ in intervals 
$I_{k}:=(q_{k},b_{k})$, that is
\[
L_{3,1}^{\ast}(b_{k})-L_{3,1}^{\ast}(q_{k})= (b_{k}-q_{k})(-\frac{1}{3}+\varepsilon),
\qquad L_{3,2}^{\ast}(b_{k})-L_{3,2}^{\ast}(q_{k})= (b_{k}-q_{k})(-\frac{1}{3}+\varepsilon),
\]
it follows from \eqref{eq:lsumme} that 
the sum $L_{3,3}^{\ast}+L_{3,4}^{\ast}$ asymptotically increases with constant
slope $2/3$ in $I_{k}$, that is
\[
L_{3,3}^{\ast}(b_{k})+L_{3,4}^{\ast}(b_{k})-L_{3,3}^{\ast}(q_{k})-L_{3,4}^{\ast}(q_{k})= 
(b_{k}-q_{k})(\frac{2}{3}+\varepsilon). 
\]
Consequently, if we can show that both $L_{3,3}^{\ast}$ and $L_{3,4}^{\ast}$
increase at most by $1/3$ in any large subinterval of $I_{k}$, that is for any $q_{k}\leq a<b\leq b_{k}$ we have
\begin{equation} \label{eq:exakter}
L_{3,3}^{\ast}(b)-L_{3,3}^{\ast}(a)\leq (b-a)(\frac{1}{3}+\varepsilon), \qquad 
L_{3,4}^{\ast}(b)-L_{3,4}^{\ast}(a)\leq (b-a)(\frac{1}{3}+\varepsilon),
\end{equation}
then both must have asymptotically constant increase by precisely $1/3$ in the entire interval $I_{k}$,
i.e. equality in \eqref{eq:exakter}.
We more precisely show the following claims. {\em Claim A:}
For any parameter $\tilde{X}\in(H(P_{k}),\infty)$, let 
\[
U_{k,\tilde{X}}=P_{k}\cdot E_{t}, \qquad V_{k,\tilde{X}}=P_{k}\cdot E_{t+1}
\]
with $t=t(k,\tilde{X})$ chosen as the largest integer such that 
$\max\{ H(U_{k,\tilde{X}}), H(V_{k,\tilde{X}})\}\leq \tilde{X}$. Then the functions $L_{3,.}^{\ast}(q)$ arising
from the succession (equals the pointwise minimum)
of the $L_{U_{k,\tilde{X}}}^{\ast}, L_{V_{k,\tilde{X}}}^{\ast}$ as 
$\tilde{X}$ runs through $(H(P_{k}),\infty)$
via \eqref{eq:incide} have asymptotically constant slope $1/3$ in $(b_{k-1},\infty)$.
By this more precisely we mean that for any $b_{k-1}\leq \tilde{X}<\tilde{Y}$ 
if $(a,L_{U_{k,\tilde{X}}}(a))$ or $(a,L_{V_{k,\tilde{X}}}(a))$ lies in the graph of 
$L_{U_{k,\tilde{X}}}$ or $L_{V_{k,\tilde{X}}}$ respectively
and similarly for $(b,L_{U_{k,\tilde{X}}}(b))$ or $(b,L_{V_{k,\tilde{X}}}(b))$, then we have
\[
L_{U_{k,\tilde{Y}}}^{\ast}(b)-L_{U_{k,\tilde{X}}}^{\ast}(a)
= (b-a)(\frac{1}{3}+\varepsilon), \qquad
L_{V_{k,\tilde{Y}}}^{\ast}(b)-L_{V_{k,\tilde{X}}}^{\ast}(a)= (b-a)(\frac{1}{3}+\varepsilon)
\]
{\em Claim B:}
Moreover if we restrict to $\tilde{X}\in(H(P_{k+1}),H(P_{k+2}))$, then the functions
$L_{U_{k,\tilde{X}}}^{\ast}$ and $L_{V_{k,\tilde{X}}}^{\ast}$ induce 
$L_{3,3}^{\ast}$ and $L_{3,4}^{\ast}$ on $I_{k}$ respectively
(remark:  as we will see later on they induce $L_{3,1}^{\ast}$ and $L_{3,2}^{\ast}$ in 
intervals $(b_{k-1},q_{k})$ if we let $\tilde{X}\in(H(P_{k}),H(P_{k+1}))$).

First recall that at the beginning $q_{k}$ of the interval $I_{k}$ the successive
minima are induced basically by $\{ P_{k}, TP_{k}, P_{k+1},TP_{k+1}\}$. 
Claim A follows basically directly from Lemma~\ref{technisch}, where $E_{t}$ and $E_{t+1}$
respectively play the role of $Q$ and $P_{k}$ the role of $P$. Note also that
$\delta$ from Lemma~\ref{technisch} tends to $0$ in our context in view of \eqref{eq:unibaer},
which also implies that the minima (in fact the entire functions) of consecutive functions of the form
$L_{U_{k,\tilde{X}}}^{\ast}$ or $L_{V_{k,\tilde{X}}}^{\ast}$ do not differ much.
Finally it should be pointed out that the condition $1/\log H(Q)=O(\delta)$ does
not cause problems since for any fixed $\delta>0$ and smaller heights $H(Q)$ only
minor changes of the function $L_{3,.}^{\ast}(q)$ can appear in intervals $(b_{k-1}, b_{k-1}+O(1))$, 
such that the global behavior of the function is not affected. 
For Claim B further observe that $\{ U_{k,\tilde{X}},V_{k,\tilde{X}}\}$ span the same space as
$\{ P_{k},TP_{k}\}$ for all $\tilde{X}\in(H(P_{k}),\infty)$, and we have already noticed
that polynomials in the space $\{ P_{k+1},TP_{k+1}\}$ induce the first two successive minima in $I_{k}$ and
$\{ P_{k},TP_{k},P_{k+1},TP_{k+1}\}$ are linearly independent. Hence
$L_{3,3}^{\ast}$ and $L_{3,4}^{\ast}$ are bounded above by $L_{U_{k,\tilde{X}}}^{\ast}$ and $L_{V_{k,\tilde{X}}}^{\ast}$
in $I_{k}$ respectively, and thus each increase at most by $1/3$. 
As noticed above we may conclude $L_{3,3}^{\ast}$ and $L_{3,4}^{\ast}$ must actually 
coincide with the functions induced by $L_{U_{k,\tilde{X}}}^{\ast}$ and $L_{V_{k,\tilde{X}}}^{\ast}$ respectively.

Thus together with \eqref{eq:brauche} we have proved
\begin{equation} \label{eq:raucher}
\lim_{k\to\infty} \psi_{3,1}^{\ast}(Q_{k})= \lim_{k\to\infty} \psi_{3,2}^{\ast}(Q_{k})
=\lim_{k\to\infty} -\psi_{3,3}^{\ast}(Q_{k})= \lim_{k\to\infty} -\psi_{3,4}^{\ast}(Q_{k})=
\frac{1-\sqrt{5}}{3(3+\sqrt{5})}.  
\end{equation}
We show next that in the interval $J_{k}:=(b_{k},q_{k+1})$ the functions $L_{3,1}^{\ast}, L_{3,2}^{\ast}$
have slope $-1/3$ whereas the functions $L_{3,3}^{\ast}, L_{3,4}^{\ast}$ have (asymptotic) slope $1/3$ until they all meet (asymptotically) at $q_{k+1}$. More precisely
\[
L_{3,1}^{\ast}(q_{k+1})-L_{3,1}^{\ast}(b_{k})=(q_{k+1}-b_{k})(-\frac{1}{3}+\varepsilon), \quad 
L_{3,2}^{\ast}(q_{k+1})-L_{3,2}^{\ast}(b_{k})=(q_{k+1}-b_{k})(-\frac{1}{3}+\varepsilon)
\]
such as
\[
L_{3,3}^{\ast}(q_{k+1})-L_{3,3}^{\ast}(b_{k})=(q_{k+1}-b_{k})(\frac{1}{3}+\varepsilon), \quad 
L_{3,4}^{\ast}(q_{k+1})-L_{3,4}^{\ast}(b_{k})=(q_{k+1}-b_{k})(\frac{1}{3}+\varepsilon)
\]
and
\[
L_{3,4}^{\ast}(q_{k+1})-L_{3,1}^{\ast}(q_{k+1})\leq \varepsilon q_{k+1}.
\]
Again by \eqref{eq:virtus} with index shift $k$ to $k+1$ 
we know that for arbitrarily small $\varepsilon$ and all large $k\geq k_{0}(\varepsilon)$ 
we indeed have
\begin{equation} \label{eq:virtue}
\vert L_{3,i}^{\ast}(q_{k+1})\vert\leq \varepsilon q_{k+1}, \qquad 1\leq i\leq 4.
\end{equation}
Since we have shown that $L_{3,1}^{\ast}$ and $L_{3,2}^{\ast}$ decay in $I_{k}$ with slope
$-1/3$ and \eqref{eq:raucher} holds it suffices to show that $J_{k}$ has asymptotically 
equal length as $I_{k}$, that is $\lim_{k\to\infty} \vert J_{k}\vert/\vert I_{k}\vert=1$,
to conclude that $L_{3,3}^{\ast}$ and $L_{3,4}^{\ast}$ must
decay with the minimum possible slope $-1/3$ in the entire interval $J_{k}$ and more precisely
\begin{equation}  \label{eq:show}
\lim_{k\to\infty} -\frac{L_{3,1}^{\ast}(b_{k})}{q_{k+1}-b_{k}}=
\lim_{k\to\infty} -\frac{L_{3,2}^{\ast}(b_{k})}{q_{k+1}-b_{k}}= 
\lim_{k\to\infty} \frac{L_{3,3}^{\ast}(b_{k})}{q_{k+1}-b_{k}}= 
\lim_{k\to\infty} \frac{L_{3,4}^{\ast}(b_{k})}{q_{k+1}-b_{k}}= 
 \frac{1}{3}.
\end{equation}
We show the claim that $I_{k}$ and $J_{k}$ have asymptotically equal length,
that is $\vert I_{k}\vert/\vert J_{k}\vert=1+o(1)$ as $k\to\infty$.
By construction this is equivalent to $b_{k}$ being asymptotically equal to $(q_{k}+q_{k+1})/2$,
that is $b_{k}=(q_{k}+q_{k+1})/2+o(q_{k})$.
Since $\lim_{k\to\infty} \log H(P_{k+1})/\log H(P_{k})=\nu$ and 
$L_{3,.}^{\ast}(q_{k})=o(q_{k})$ and $L_{3,.}^{\ast}(q_{k+1})=o(q_{k+1})$ as $k\to\infty$.
Further notice that $L_{3,1}^{\ast}, L_{3,2}^{\ast}$ decay in $(q_{k},b_{k})$ induced by $P_{k+1},TP_{k+1}$ and
thus by \eqref{eq:incide} we have $L_{3,j}^{\ast}(q_{i})=\log H(P_{i+1})-q_{i}/3+O(1)$ for
$1\leq j\leq 4$ and all $i\geq 1$. Putting all together leads to
\begin{equation} \label{eq:quotiente}
\lim_{k\to\infty} \frac{q_{k+1}}{q_{k}}=\nu.
\end{equation}
Thus the claimed asymptotic relation $b_{k}=(q_{k}+q_{k+1})/2+o(q_{k})$ 
is equivalent to $b_{k}=q_{k}\cdot(1+\nu)/2+o(q_{k})$.
We know that at $Q_{k}=e^{b_{k}}$ we have asymptotically 
\begin{equation} \label{eq:nebu}
\psi_{3,3}^{\ast}(Q_{k})=\frac{b_{k}-q_{k}}{3}+o(q_{k}), \qquad k\to\infty,
\end{equation}
since $L_{3,3}^{\ast}$ and $L_{3,3}^{\ast}$ are small at $q_{k}$ by \eqref{eq:virtus} and 
rise with slope $1/3$ in $I_{k}$. We remark that the asymptotic \eqref{eq:nebu} holds for $\psi_{3,4}^{\ast}(Q_{k})$ as well.
On the other hand \eqref{eq:raucher} provides an asymptotic formula for $\psi_{3,3}^{\ast}(Q_{k})$
and $\psi_{3,4}^{\ast}(Q_{k})$.
It follows directly from the definition of $L_{3,j}^{\ast}$ via $\psi_{3,j}^{\ast}$ in \eqref{eq:lpsi} that $\psi_{3,3}^{\ast}(Q_{k})$ is the slope 
from the origin to $(b_{k},L_{3,3}^{\ast}(b_{k}))$
of $L_{3,3}^{\ast}$ in the Schmidt-Summerer diagram (and similarly for $L_{3,4}^{\ast}$). 
Hence asymptotically 
\begin{equation} \label{eq:loes}
\psi_{3,3}^{\ast}(Q_{k})=\psi_{3,4}^{\ast}(Q_{k})=\frac{\sqrt{5}-1}{3(3+\sqrt{5})}b_{k}+o(b_{k}), \qquad k\to\infty.
\end{equation}
Again the asymptotic \eqref{eq:loes} holds for $\psi_{3,4}^{\ast}(Q_{k})$ as well.
Comparing the two expressions for $\psi_{3,3}^{\ast}(Q_{k})$ in \eqref{eq:nebu} and 
\eqref{eq:loes}, with a short computation indeed we verify $b_{k}=q_{k}\cdot(1+\nu)/2+o(q_{k})$,
so we have proved that $I_{k}$ and $J_{k}$ have asymptotically equal length.

Since consequently $L_{3,3}^{\ast}$ and $L_{3,4}^{\ast}$ both 
asymptotically decay with slope $-1/3$ in 
$J_{k}$, from \eqref{eq:lsumme} again we deduce that the sum $L_{3,1}^{\ast}+L_{3,2}^{\ast}$
must asymptotically increase by $2/3$ in $J_{k}$. Now recall in Claim A
we showed that $L_{U_{k,\tilde{X}}}^{\ast}, L_{V_{k,\tilde{X}}}^{\ast}$
asymptotically induce an increase with slope at most $1/3$ in the entire interval $(b_{k-1},\infty)$
if we let $\tilde{X}$ run through $(H(P_{k}),\infty)$.
Hence if we restrict to $\tilde{X}\in(H(P_{k}),H(P_{k+1}))$,
by a very similar argument as in Claim B, in the interval $(b_{k-1},q_{k})$ they induce  
$L_{3,1}^{\ast}$ and $L_{3,2}^{\ast}$ such that they both asymptotically increase precisely with this slope $1/3$. 
By index shift the analogue claim is clearly also true for $(b_{k},q_{k+1})=J_{k}$. 
Hence indeed both $L_{3,1}^{\ast}$ and $L_{3,2}^{\ast}$ must asymptotically increase with 
slope precisely $1/3$ in the entire interval $J_{k}$. 

Observe that the end of $J_{k}$ is 
the beginning of $I_{k+1}$, such that we have basically established a complete
description of all functions $L_{3,1}^{\ast},\ldots,L_{3,4}^{\ast}$ on $(0,\infty)$.
The characterizations of the graphs of $L_{3,i}^{\ast}(q)$ established above show that 
asymptotically at the values $q=b_{k}$
both the smallest local minima of $L_{3,1}^{\ast}(q), L_{3,2}^{\ast}(q)$ 
(in sense of minimal values of $\psi_{3,1}^{\ast}(Q), \psi_{3,2}^{\ast}(Q)$) and the largest local maxima of 
$L_{3,3}^{\ast}(q), L_{3,4}^{\ast}(q)$ (in sense of maximal values of $\psi_{3,3}^{\ast}(Q),\psi_{3,4}^{\ast}(Q)$)
are attained. Moreover both 
$\vert L_{3,1}^{\ast}(b_{k})-L_{3,2}^{\ast}(b_{k})\vert$ and 
$\vert L_{3,3}^{\ast}(b_{k})-L_{3,4}^{\ast}(b_{k})\vert$ are bounded uniformly in $k$,
in fact more generally $\vert L_{3,1}^{\ast}(q)-L_{3,2}^{\ast}(q)\vert$ and 
$\vert L_{3,3}^{\ast}(q)-L_{3,4}^{\ast}(q)\vert$ are uniformly bounded for $q\in(0,\infty)$.
Thus with \eqref{eq:raucher} we have
\[
\underline{\psi}_{3,1}^{\ast}= \underline{\psi}_{3,2}^{\ast}=\frac{1-\sqrt{5}}{3(3+\sqrt{5})}, \qquad
\overline{\psi}_{3,3}^{\ast}= \overline{\psi}_{3,4}^{\ast}=\frac{\sqrt{5}-1}{3(3+\sqrt{5})}.
\]
With \eqref{eq:jaja}, \eqref{eq:umrechnen} and \eqref{eq:umrechnen2} we derive 
\begin{equation} \label{eq:dazu}
w_{3}(\zeta)=w_{3,2}(\zeta)=\rho, \qquad \lambda_{3}(\zeta)=\lambda_{3,2}(\zeta)=\frac{1}{\sqrt{5}}.
\end{equation}
This contains in particular the claims in \eqref{eq:zweite}. 
\end{proof}

\begin{remark} \label{dasreh}
We can also determine the remaining constants 
$w_{3,i}, \lambda_{3,i},\widehat{w}_{3,i},\widehat{\lambda}_{3,i}$ for extremal numbers.
From \eqref{eq:debre} and \eqref{eq:dazu} we deduce
\begin{equation} \label{eq:hard}
\widehat{w}_{3,3}(\zeta)=\widehat{w}_{3,4}(\zeta)=\sqrt{5}, \qquad
\widehat{\lambda}_{3,3}(\zeta)=\widehat{\lambda}_{3,4}(\zeta)=\frac{1}{\rho}.
\end{equation}
Moreover the above characterizations of the functions $L_{3,i}^{\ast}$ imply
\[
\overline{\psi}_{3,1}^{\ast}= \overline{\psi}_{3,2}^{\ast}=
\underline{\psi}_{3,1}^{\ast}= \underline{\psi}_{3,2}^{\ast}=0.
\] 
With \eqref{eq:umrechnen2} and \eqref{eq:debre} this is equivalent to
\begin{align}  \label{eq:beweis}
w_{3,3}(\zeta)&=w_{3,4}(\zeta)=\widehat{w}_{3}(\zeta)=\widehat{w}_{3,2}(\zeta)=3, \\
\lambda_{3,3}(\zeta)&=\lambda_{3,4}(\zeta)=
\widehat{\lambda}_{3}(\zeta)=\widehat{\lambda}_{3,2}(\zeta)=\frac{1}{3}. \nonumber
\end{align}

\end{remark}

The description of the combined graph of the functions $L_{3,j}^{\ast}(q)$ 
and the information on the structure of the polynomials inducing them
from the proof of Theorem~\ref{ndrei} allows for estimating 
the approximation to an extremal number by algebraic numbers of degree precisely three.

\begin{proof}[Proof of Theorem~\ref{sternchen}]
It follows from the proof of Theorem~\ref{ndrei} and the description above that 
the first two successive minima functions
of the linear form problem related to $\psi_{3,1}^{\ast}, \psi_{3,2}^{\ast}$ are induced by
polynomial multiples of $P_{k}$ from Theorem~\ref{roythm}, and for each $k$ these multiples
span the same space as $\{ P_{k},TP_{k}\}$. Since $P_{k}$ have degree
two there is no irreducible polynomial of degree three which lies in the space spanned by
$\{ P_{k},TP_{k}\}$ for some $k$. Thus the optimal exponent in \eqref{eq:trio}
is not larger than $w_{3,3}(\zeta)$.
On the other hand it was shown in the proof of Theorem~\ref{ndrei} that $w_{3,3}(\zeta)=3$, 
see \eqref{eq:beweis}. Thus, combining these facts, we see
that indeed \eqref{eq:trio} has only finitely many solutions in $Q\in{\mathbb{Z}[T]}$ 
an irreducible polynomial of degree precisely three.
From \eqref{eq:trio} we infer \eqref{eq:genaudrei1} by a standard argument. 
Indeed if $R$ is the minimal
polynomial of some $\alpha$ then $\vert R(\zeta)\vert=\vert R(\zeta)-R(\alpha)\vert=
\vert \zeta-\alpha\vert\cdot R^{\prime}(z)$ for some $z$ between $\alpha$ and $\zeta$ by intermediate
theorem of differentiation. On the other hand $\vert R^{\prime}(z)\vert \ll H(R)$ for bounded $z$ is easy to see,
and the claim \eqref{eq:genaudrei1} follows from \eqref{eq:trio}.

Next we show \eqref{eq:trio2} and \eqref{eq:genaudrei2}.
By essentially the argument from the proof of \eqref{eq:trio} again $\widehat{w}_{3,3}(\zeta)$ is 
an upper bound for the exponent in \eqref{eq:trio2} for some large $X$. 
On the other hand we have noticed in \eqref{eq:hard} that
$\widehat{w}_{3,3}(\zeta)=\widehat{w}_{3,4}(\zeta)=\sqrt{5}$. 
Combination yields \eqref{eq:trio2} and we deduce \eqref{eq:genaudrei2} from it very similarly
as \eqref{eq:genaudrei1} from \eqref{eq:trio}.

For \eqref{eq:otherhand} recall that in the proof of Theorem~\ref{ndrei} we showed that
for any large $k$ there exists a linear polynomial $E_{l}$ such that with 
$X:=H(P_{k+1})$ and $Q_{k,1}:= P_{k}E_{l}$ we have
\begin{align}
H(Q_{k,1})&\asymp H(P_{k+1})=X,  \nonumber \\
\vert Q_{k,1}(\zeta)\vert &\leq X^{-3+\epsilon}, \qquad \vert P_{k+1}(\zeta)\vert \leq X^{-3+\epsilon}. \label{eq:number}
\end{align}
Since $Q_{k,1}$ is not irreducible by construction and $P_{k+1}$ has degree only $2$,
we consider the polynomials $S_{k,j}(T):= Q_{k,1}(T)+jT\cdot P_{k+1}(T)$ for $j\in\{1,2\}$.
We show that at least one of these two polynomials has the desired properties (in fact we
need the distinction only for the right hand side of \eqref{eq:otherhand}, the left follows for both $j=1$ and $j=2$).
The polynomials $S_{k,j}(T)$ obviously have degree three and height $H(S_{k,j})\ll X$.
Moreover with \eqref{eq:number} we infer
\begin{equation} \label{eq:comeon}
\vert S_{k,j}(\zeta)\vert= \vert Q_{k,1}(\zeta)+j\zeta P_{k+1}(\zeta)\vert
\leq \vert Q_{k,1}(\zeta)\vert+j\vert \zeta\vert\cdot \vert P_{k+1}(\zeta)\vert
\ll_{\zeta} X^{-3+\epsilon}, \quad 1\leq j\leq 2.
\end{equation}
Next we check that $S_{k,j}$ are irreducible for large $k$ and $1\leq j\leq 2$. 
Consider $j$ fixed and suppose $S_{k,j}$ is reducible.
Then we may write $S_{k,j}(T)=M(T)N(T)$ for $M,N\in\mathbb{Z}[T]$ each of degree one or two. Then
$\vert S_{k,j}(\zeta)\vert=\vert M(\zeta)\vert\cdot \vert N(\zeta)\vert$ and it follows
from \eqref{eq:multipli} and \eqref{eq:comeon} that at least one of the inequalities
\[
\vert M(\zeta)\vert \leq H(M)^{-3+2\epsilon}, \qquad \vert N(\zeta)\vert \leq H(N)^{-3+2\epsilon}
\]
must be satisfied, see also the remark subsequent to \eqref{eq:multipli}. 
Without loss of generality say this holds for $M$.
However, since $w_{2,2}(\zeta)\leq \tau<3$, see \eqref{eq:rhotau},
and $M$ has degree at most two, it follows from Theorem~\ref{roythm} that the inequality
can only be satisfied if $M$ is some $P_{l}$ from Theorem~\ref{roythm}.
However, by construction of $S_{k,j}$ we clearly cannot have $P_{k}\vert S_{k,j}$ or $P_{k+1}\vert S_{k,j}$.
Thus $M=P_{l}$ for some $l\leq k-1$. Theorem~\ref{roythm} further implies
\[
H(M)\ll H(P_{k-1})\ll H(P_{k+1})\cdot \frac{H(P_{k-1})}{H(P_{k+1})}
= X\cdot \frac{H(P_{k-1})}{H(P_{k+1})}\ll X^{1/\nu^{2}}=
X^{1/\tau}
\]
and it follows further that
\begin{equation} \label{eq:mgl}
\vert M(\zeta)\vert\asymp H(M)^{-\rho}\gg X^{-\rho/\tau}=X^{-\nu}.
\end{equation}
Since $M=P_{l}$ has degree two and $S_{k,j}$ degree three,
the polynomial $N$ must have degree one such that by $\lambda_{1}(\zeta)=1$ from 
\eqref{eq:werte} we have
\begin{equation} \label{eq:ngl}
\vert N(\zeta)\vert \gg H(N)^{-1-\epsilon}\gg X^{-1-\epsilon}.
\end{equation}
Combination of \eqref{eq:mgl} and \eqref{eq:ngl} yields
\[
\vert S_{k,j}(\zeta)\vert = \vert M(\zeta)\vert\cdot \vert N(\zeta)\vert
\gg X^{-\nu-1-\epsilon}=X^{-\tau-\epsilon}.
\]
Again from $\tau<3$ we obtain a contradiction to \eqref{eq:comeon} for small $\epsilon$. 
Hence the assumption was wrong and indeed $S_{k,j}$ must be irreducible for $j\in\{ 1,2\}$, 
and in view of \eqref{eq:comeon} we have finished the proof of the left hand side of \eqref{eq:otherhand}.

For the right hand side of \eqref{eq:otherhand} suppose we have already shown
that for all large $k$ and some $j=j(k)\in\{1,2\}$ we have
\begin{equation} \label{eq:lugner}
\vert S_{k,j}^{\prime}(\zeta)\vert \gg X^{1-\epsilon}.
\end{equation}
Then the claim follows together with \eqref{eq:comeon}
from the left hand side for $\alpha$ some root of the corresponding $S_{k,j}$
by a similar standard argument as in the deduction of \eqref{eq:genaudrei1}
from \eqref{eq:trio}. Indeed it is well-known that any polynomial $U\in\mathbb{Z}[T]$ has a root $\beta$ that satisfies $\vert \beta-\zeta\vert\ll \vert U(\zeta)\vert/H(U)$, see for
example~\cite{royyy}. The claim follows with $U=S_{k,j}$.
It remains to be checked that \eqref{eq:lugner} holds, for which we use \eqref{eq:strich}. 
First note that the derivative of $S_{k,j}$ can be written
\begin{equation} \label{eq:rrr}
\vert S_{k,j}^{\prime}(\zeta)\vert= \vert Q_{k,1}^{\prime}(\zeta)+jP_{k+1}(\zeta)+j\zeta P_{k+1}^{\prime}(\zeta)\vert,
\qquad 1\leq j\leq 2.
\end{equation}
Obviously the term $jP_{k+1}(\zeta)$ in the sum is negligible since it is very small.
Hence \eqref{eq:rrr} can be small only if $Q_{k,1}^{\prime}(\zeta)$ is of the same order (and reverse sign)
as $j\zeta P_{k+1}^{\prime}(\zeta)$. On the other hand \eqref{eq:strich} implies for all large $k$ the estimate
\[
\vert j\zeta P_{k+1}^{\prime}(\zeta)\vert \geq j\vert\zeta\vert H(P_{k+1})^{1-\epsilon}\gg_{\zeta} X^{1-\epsilon},
\qquad j\in\{1,2\},
\]
and very similarly the difference between the right hand sides in \eqref{eq:rrr}
for $j=2$ and $j=1$ is at least of order $X^{1-\epsilon}$ as well. 
It follows that \eqref{eq:lugner} can be violated for at most one index $j\in\{1,2\}$, 
and for the other index \eqref{eq:lugner} must be satisfied. 
This finishes the proof of \eqref{eq:otherhand}.
\end{proof}

We finish by giving a heuristic argument why the exponents in \eqref{eq:trio2} and \eqref{eq:genaudrei2}
should be optimal as well. For any $\tilde{X}$ we can again consider linear combinations 
$S_{k,j}(T)=jTP_{k+1}(T)+P_{k}(T)E_{t}(T)$ for $k=k(\tilde{X})$ largest 
possible such that $H(P_{k+1})\leq \tilde{X}$ and some $E_{t}$ of degree one from the proof 
of Theorem~\ref{ndrei} such that \eqref{eq:trio2} is satisfied for $Q(T)=Q_{k,1}(T)=P_{k}(T)E_{t}(T)$. 
Given the irreducibility of $S_{k,j}$ for all large $k$
and $j$ rather small, we can again basically proceed as in the proof of \eqref{eq:otherhand}.
However, the method from the proof of \eqref{eq:otherhand} to guarantee the
irreducibility of some of the arising $S_{k,j}(T)$ does not work here.

\vspace{1cm}

The author warmly thanks the anonymous referee for the careful reading 
and for pointing out inaccuracies

\end{document}